\documentclass[10pt]{amsart}

\usepackage{amssymb}

\usepackage{enumerate}

\usepackage{url}
\usepackage{lipsum}

\makeatletter
\@namedef{subjclassname@2010}{%
  \textup{2010} Mathematics Subject Classification}
\makeatother

\newtheorem{thm}{Theorem}[section]

\newtheorem{lem}[thm]{Lemma}
\newtheorem{prop}[thm]{Proposition}

\theoremstyle{definition}

\newtheorem{rem}[thm]{Remark}

\numberwithin{equation}{section}

\begin{document}

\title[Less than one implies zero]{Less than one implies zero}

\author[F. L. Schwenninger]{Felix L. Schwenninger}
\address{Dept. of Applied Mathematics, \\ University of Twente, P.O. Box 217, \\ 7500 AE Enschede, The Netherlands}
\email{f.l.schwenninger@utwente.nl}
\thanks{The first author has been supported by the Netherlands Organisation for Scientific Research (NWO), grant no.\ 613.001.004.}
\author[H. Zwart]{Hans Zwart}
\address{Dept. of Applied Mathematics, \\ University of Twente, P.O. Box 217, \\ 7500 AE Enschede, The Netherlands}
 \email{h.j.zwart@utwente.nl}

\date{13 February 2015}

\begin{abstract}
In this paper we show that from the estimate $\sup_{t \geq 0}\|C(t) - \cos(at)I\| <1$ we can conclude that $C(t)$ equals $\cos(at) I$. Here $\left(C(t)\right)_{t \geq 0}$ is a strongly continuous cosine family on a Banach space. 
 \end{abstract}



\maketitle

\section{Introduction}

Let $\left( T(t) \right)_{t \geq 0}$ denote a strongly continuous semigroup on the Banach space $X$ with infinitesimal generator $A$. 
It is well-known that the inequality
\begin{equation}\label{1-0law}
\limsup_{t\to0^{+}}\|T(t) - I\| <1,
\end{equation}
implies that the generator $A$ is a bounded operator, see e.g.\ \cite[Remark 3.1.4]{Staffans}.  That the stronger assumption of having
\begin{equation}
  \label{eq:1}
  r:=\sup_{t\geq0}\|T(t) - I\| <1,
\end{equation}
implies that $T(t)=I$ for all $t\geq 0$ seems not to be equally well-known among researchers working in the area of strongly continuous semigroup. The result was proved in the sixties, see e.g.\ Wallen \cite{Wallen} and Hirschfeld \cite{Hirschfeld}. We refer the reader to \cite[Lemma 10]{BobrowskiApprox} for a more detailed listing of related references. 
In this paper we investigate a similar question for cosine families $\left( C(t) \right)_{t \geq 0}$. Recently,  Bobrowski and Chojnacki showed in \cite[Theorem 4]{BobrowskiApprox}  that 
\begin{equation}
\label{eq:1a}
  \sup_{t\geq0}\|C(t)-\cos(at)I\| < \frac{1}{2},
\end{equation}
implies $C(t) = \cos(at)I$ for all $t \geq 0$. They used this to conclude that scalar cosine families are isolated points within the space of bounded strongly continuous cosine families acting on a fixed Banach space, equipped with the supremum norm. \\
The purpose of this note is to extend the result of \cite{BobrowskiApprox} by showing that the half in (\ref{eq:1a}) may be replaced by one. 
More precisely, we prove the following.
\begin{thm}
\label{mainresult}
Let $\left(C(t)\right)_{t\geq 0}$ be a strongly continuous cosine family on the Banach space $X$ and let $a \geq 0$.
  If the following inequality holds for $r<1$,
\begin{equation}
  \label{eq:10a1}
  	\sup_{t\geq0}\|C(t)- \cos(at) I  \| < r,
\end{equation}
then $C(t)=\cos(at)I$.
\end{thm}
Between the first draft\footnote{see F.~Schwenninger, H.~Zwart, \textit{Less than one, implies zero}, \url{http://arxiv.org/abs/1310.6202v1.pdf}.} and this version of the manuscript, Chojnacki showed in \cite{ChojnackiOneTwo} that Theorem \ref{mainresult} even holds for cosine families on normed algebras indexed by general abelian groups and without assuming strong continuity. Furthermore,  Bobrowski, Chojnacki and Gregosiewicz \cite{BobrowskiZeroTwo} and, independently, Esterle \cite{Esterle2015} extended Theorem \ref{mainresult} to $r<\frac{8}{3\sqrt{3}}\approx 1.54$. This is optimal as $\sup_{t\geq0}|\cos(3t)-\cos(t)|=\frac{8}{3\sqrt{3}}$.
Again, their results do not require the strong continuity assumption and hold for cosine families on general normed algebras with a unity element.\\
Let us remark that the case $a=0$ is special. In a \textit{three-line-proof} \cite{UlmerSeminare2012}, Arendt showed that $r<\frac{3}{2}$ still implies that $C(t)=I$ for all $t\geq0$ then. In \cite{SchwenningerZwartZeroTwo}, we proved that for $\left(C(t)\right)_{t\geq 0}$ strongly continuous,
\begin{equation}\label{eq:ZeroTwoSup}
	\sup_{t\geq0}\|C(t)-I\|<2\ \implies C(t)=I\ \forall t\geq0.
\end{equation}
Moreover, we were able to show the following \textit{zero-two law},
\begin{equation}\label{eq:ZeroTwo}
	\limsup_{t\to0^{+}}\|C(t)-I\|<2\ \implies \lim_{t\geq0}\|C(t)-I\|=0,
\end{equation}
which can be seen as the cosine families version of (\ref{1-0law}). Recently, Chojnacki \cite{ChojnackiZeroTwo} also extended (\ref{eq:ZeroTwoSup}) and  (\ref{eq:ZeroTwo}), allowing for, not necessarily strongly continuous, cosine familes on general normed algebras with a unity element.\\
In the next section we prove Theorem \ref{mainresult} for $a\neq0$ using elementary techniques, which seem to be less involved than the technique used in \cite{BobrowskiApprox}. As mentioned, the case $a=0$ can be found in \cite{SchwenningerZwartZeroTwo}, see also \cite{BobrowskiApprox,BobrowskiZeroTwo,ChojnackiZeroTwo,ChojnackiOneTwo}.
%As will become clear, the proof of the second item in Theorem \ref{mainresult} is elementary once the spectrum of  the infinitesimal generator is characterized. The proof of the first item is longer, but only 
% relies on elementary techniques. T
%
%The case $a=0$ can be considered similar as in the semigroup case before. Moreover, it is also included in \cite[eq. (4.1)]{BobrowskiApprox}, where it is in turn a consequence of the semigroup result.

\section{Proof of Theorem \ref{mainresult}}
Let $\left(C(t)\right)_{t\geq 0}$ be a strongly continuous cosine family on the Banach space $X$ with infinitesimal generator of $A$ with domain $D(A)$.  For an introduction to cosine families we refer to e.g.\ \cite{ABHN, Fattorini69I}. Assume that for some $r>0$
\begin{equation}
  \label{eq:10a}
  	\sup_{t\geq0}\|C(t)- \cos(at) I  \| = r.
\end{equation}
If $a>0$ we may apply scaling on $t$. Hence in that situation, we can take without loss of generality $a=1$, thus
\begin{equation}\label{eq:10}
 	 \sup_{t\geq0}\|C(t)- \cos(t) I \|= r.
\end{equation}
%In the proofs we will often work with this equation. 
The following lemma is essential in proving Theorem \ref{mainresult}.
\begin{lem}\label{le:spectrum}
	Let $\left(C(t) \right)_{t\geq 0}$ be a cosine family such that (\ref{eq:10a}) holds for $r<1$ and $a\geq0$. Then, the spectrum of its generator $A$ satisfies $\sigma(A)\subseteq\left\{-a^{2}\right\}$.
\end{lem}
\begin{proof}
The case $r=0$ is trivial, thus let $r>0$.
\mbox{}From (\ref{eq:10a}) it follows in particular that the cosine family $\left(C(t)\right)_{t\geq 0}$ is bounded. Using Lemma 5.4 from \cite{Fattorini69I} we conclude that for every $s \in {\mathbb C}$ with positive real part $s^{2}$ lies in the resolvent set of $A$, i.e., $s^2 \in\rho(A)$. Thus the spectrum of $A$ lies on the non-positive real axis.
\newline To determine the spectrum, we use the following identity, see \cite[Lemma 4]{Nagy74Sz}. For $\lambda\in\mathbb{C}$, $s\in\mathbb{R}$ and $x\in D(A)$ there holds
\begin{equation*}
	\frac{1}{\lambda}\int_{0}^{s}\sinh(t-s)C(t)(\lambda^{2}I-A)x \ dt = (\cosh(\lambda s)I- C(s))x.
\end{equation*}
By this and the definition of the approximate point spectrum,
\begin{equation*}
	\sigma_{ap}(A)=\left\{\lambda\in\mathbb{C}\mid \exists(x_{n})_{n\in\mathbb{N}}\subset D(A),\|x_{n}\|=1,\lim_{n\to\infty}\|(A-\lambda I)x_{n}\|=0\right\},
\end{equation*}
it follows that if $\lambda^{2}\in\sigma_{ap}(A)$, then $\cosh(\lambda s)\in\sigma_{ap}(C(s))$. Hence,
\begin{equation}\label{eq:factNagy}
	\cosh\left(s\sqrt{\sigma_{ap}(A)}\right)\subset \sigma_{ap}(C(s)),\qquad \forall s\in\mathbb{R}.
\end{equation}
Since $\sigma(A)\subset\mathbb{R}_{0}^{-}$, the boundary of the spectrum equals the spectrum. Combining this with the fact that the boundary of the spectrum is contained in the approximate point spectrum, we see that $\sigma(A)=\sigma_{ap}(A)$. Let $-\lambda^{2}\in\sigma(A)$ for $\lambda\geq0$. Then, by (\ref{eq:factNagy}),
\begin{equation*}
	\cosh(\pm si\lambda)=\cos(s\lambda)\in\sigma_{ap}(C(s)), \qquad \forall s\in\mathbb{R}.
\end{equation*}
If $\lambda\neq a$, we can find $s_{0} > 0$ such that $|\cos(s_{0}\lambda)-\cos(as_{0})|\geq 1$, see Lemma \ref{le:cos}. Since $\cos(s_{0}\lambda)\in\sigma_{ap}(C(s_{0}))$, we find a sequence $(x_{n})_{n\in\mathbb{N}}\subset X$ such that $\|x_{n}\|=1$ and $\lim_{n\to\infty}\|(C(s_{0})-\cos(s_0\lambda))x_{n}\|=0$. Therefore,
\begin{equation*}
	\|(C(s_{0})-\cos(as_{0}))x_{n}\|\geq |\cos(s_{0}\lambda)-\cos(as_{0})|-\|\left(C(s_{0})-\cos(s_{0}\lambda)\right)x_{n}\|.
\end{equation*}	
Thus $\|C(s_{0})-\cos(as_{0})\| \geq 1$. This contradicts assumption (\ref{eq:10a}) as $r<1$.
\end{proof}
\begin{lem}\label{le:cos}
If $a,b\geq0$ and $a\neq b$, then $\sup_{t\geq0}|\cos(at)-\cos(bt)|>1$.
\end{lem}
\begin{proof}
If $a=0$, the assertion is clear as $\cos(\pi)=-1$. Hence, let $a,b>0$.
By scaling, it suffices to prove that
\begin{equation*}
\forall a\in(0,1)\ \exists s\geq0:\quad |\cos(as)-\cos(s)|>1.
\end{equation*}
Since $\cos(2k\pi)=1$ for $k\in\mathbb{Z}$ and $\cos(as)<0$ for $t\in\frac{\pi}{a}(\frac{1}{2}+2m,\frac{3}{2}+2m)$, $m\in\mathbb{Z}$,  we are done if we find $(k,m)\in\mathbb{Z}\times\mathbb{Z}$ such that 
\begin{equation*}
	k\in\tfrac{1}{a}\left(\tfrac{1}{4}+m,\tfrac{3}{4}+m\right).
\end{equation*}
This is equivalent to $ka-m\in(\frac{1}{4},\frac{3}{4})$. It is easy to check that for $a\in(2^{-n-1},2^{-n}]\cup[1-2^{-n-1},1-2^{-n})$ we can choose $m=2^{n-1}$ and $k=\lfloor ma\rfloor$.
\end{proof}
%\subsection{Proof of Theorem \ref{mainresult}.2}
%	Since $\left(C(t)\right)_{t \geq 0}$ is a bounded cosine family on  a Hilbert space, it follows from Fattorini, \cite[Theorem 4.1]{Fattorini70cosHS}, that $\left(C(t) \right)_{t \geq 0}$ is similar to a cosine family of self-adjoint operators. Thus there exists a bounded operator $S$, which is (boundedly) invertible such that
%	\begin{equation*}
%	\tilde{C}(t)=S^{-1}C(t)S\quad \text{ and } \quad \tilde{C}(t)^{*}=\tilde{C}(t), 
%	\end{equation*}
%	for all $t\geq0$ with $\left(\tilde{C}(t)\right)_{t \geq 0}$ a cosine family of self-adjoint operators. Furthermore, the generator of $\left(\tilde{C}(t)\right)_{t \geq 0}$, $\tilde{A}$ is self-adjoint and $\tilde{A}=S^{-1}AS$. Since $\sigma(A) \subseteq \left\{-a^{2}\right\}$ by Lemma \ref{le:spectrum}, it follows that also $\sigma(\tilde{A}) \subseteq \left\{-a^{2}\right\}$. Thus, by spectral theory for self-adjoint operators, $\sigma(\tilde{A}) = \left\{-a^{2}\right\}$ and $\tilde{A}$ has to equal $-aI$. Therefore, $A=-a^{2}I$ and $C(t)=\cos(at)I$.

%\subsection{Proof of Theorem \ref{mainresult}}

As mentioned before we may assume that $a=1$, and thus we consider equation (\ref{eq:10}) and assume that $r<1$. Hence we know that the norm of the difference $e(t)=C(t)-\cos(t)$ is uniformly below one, and we want to show that it equals zero. 
The idea is to work on the following inequality
\begin{equation}
  \label{eq:14}
  \left\| \int_{0}^{\infty} h_n(q,t) e(t) dt \right\| \leq r \int_0^{\infty} |h_n(q,t)| dt,
\end{equation}
with $h_{n}(q,t)=e^{-qt} \cos(t)^{2n+1}$, $n \in {\mathbb N}$, where $q>0$ is an auxiliary variable to be dealt with later.

Since $\left(C(t)\right)_{t \geq 0}$ is bounded, it is well-known (see e.g.\ \cite[Lemma 5.4]{Fattorini69I}) that for $s$ with $\Re(s)>0$, $s^{2}\in\rho(A)$ and we can define $E(s)$ as the Laplace transform of $e(t)$, 
\begin{equation} \label{eq:13}
  E(s):= \int_{0}^{\infty}e^{-st}e(t)\ dt=s (s^2 I - A)^{-1} -\frac{s}{s^2 +1} I
\end{equation}
To calculate the left-hand side of (\ref{eq:14}) we need the following two results. We omit the proof of the first as it can be checked by reader easily.
\begin{lem}\label{le:FScosn}
		Let $n\in\mathbb{N}$. Then, for all $t\in\mathbb{R}$,
		\begin{align*}
		\cos(t)^{2n+1}=\sum_{k=0}^{n}a_{2k+1,2n+1}\cos\left( (2k+1)t\right),%\qquad a_{2k+1}=2^{-n+1}\binom{n}{\frac{n-1-2k}{2}}
		\end{align*}
		where $a_{2k+1,2n+1}=2^{-2n}\binom{2n+1}{n-k}$.
	\end{lem}
\begin{prop}\label{prop:rep}
For 	$h_n(q,t)= -2e^{-qt} \cos(t)^{2n+1}$ and $q>0$ we have
\begin{equation*}
\int_0^{\infty} h_{n}(q,t) e(t) dt=a_{1,2n+1}\frac{g(q)}{q}I+a_{1,2n+1}B(A,q)+G(A,q),
\end{equation*}
where $a_{n}$ as in Lemma \ref{le:FScosn}, $g(q)=\frac{2q^2 +4}{(q^2+4)}$,
\begin{equation*}
B(A,q)=R\left((q+i)^{2},A\right)\ 2q\left[A-(q^{2}+1)I\right]R\left((q-i)^{2},A\right),
\end{equation*}
$R(\lambda,A) = (\lambda I - A)^{-1}$, and $G(A,q)$ is such that $\lim_{q\to0^{+}}q\cdot G(A,q)=0$ in the operator norm.
\end{prop}
\begin{proof}
By Lemma \ref{le:FScosn}, we have that
\begin{align*}
\int_0^{\infty} h_n(q,t) e(t) dt={}&-\sum_{k=0}^{n}a_{2k+1,2n+1}\ 2\int_{0}^{\infty}e^{-qt}\cos\left( (2k+1)t\right)e(t)\ dt\\					={}&-\sum_{k=0}^{n}a_{2k+1,2n+1}\left[E(q+(2k+1)i)+E(q-(2k+1)i)\right].
\end{align*}
Let us first consider the term in the sum corresponding to $k=0$. By (\ref{eq:13}), 
\begin{align}
	E(q\pm i)=\frac{q\pm i}{q(q\pm 2i)}\left[q(q\pm 2i)R((q\pm i)^{2})-I\right],
\end{align}
where $R(\lambda)$ abbreviates $R(\lambda,A)$. Hence, 
\begin{align*}
	E(q+i)+E(q-i)={}&-\frac{g(q)}{q}+(q+i)R((q+i)^{2})+(q-i)R((q-i)^{2})\\
			={}&-\frac{g(q)}{q}+R((q+i)^{2})\left[(q+i)((q-i)^{2}-A)+\right.\\
			{}&+\left.((q-i)^{2}-A)(q-i)\right]R((q-i)^{2})\\
			={}&-\frac{g(q)}{q}+R((q+i)^{2})2q\left[q^{2}I+I-A\right]R((q-i)^{2})\\
			={}&-\frac{g(q)}{q}-B(A,q).
\end{align*}
Thus, it remains to show that $qG(A,q)$ with
\begin{equation*}
	G(A,q):=-\sum_{k=1}^{n}a_{2k+1,2n+1}\left[E(q+(2k+1)i)-E(q-(2k+1)i)\right]
\end{equation*}
goes to $0$ as $q\to0^{+}$. Let $d_{k}=(2k+1)i$. By (\ref{eq:13}), %and since for $s\in\mathbb{C}^{+}$, $s^{2}\in\mathbb{C}\setminus\mathbb{R}_{0}^{-}\subset\rho(A)$,
\begin{align*}
	E(q\pm(2k+1)i)={}& (q\pm d_{k})R((q\pm d_{k})^{2})-\frac{q\pm d_{k}}{(q\pm d_{k})^{2}+1}I.
\end{align*}
Thus, since $d_{k}^{2}\in\rho(A)$ for $k\neq0$ by Lemma \ref{le:spectrum}, 
\begin{equation*}
	\lim_{q\to0^{+}}E(q\pm(2k+1)i)=\pm d_{k}R(d_{k}^{2})\pm\frac{d_{k}}{d_{k}^{2}+1},
\end{equation*}
 for $k\neq0$, hence, $\lim_{q\to0^{+}}q\cdot G(A,q)=0$. Therefore, the assertion follows.
\end{proof}
\begin{lem} \label{le:bn}
	For any $n\in\mathbb{N}$ and $a_{1,2n+1}$ chosen as in Lemma \ref{le:FScosn} holds that 
	\begin{itemize}
		\item $b_{n}:=\lim_{q\to0^{+}}q\cdot \int_{0}^{\infty}e^{-qt}|\cos(t)^{n}|\ dt$ exists and $b_{n}\geq b_{n+1}$,
		% =  \frac{1}{\pi}B\left(\frac{1}{2},\frac{n+1}{2}\right),%=\frac{\left(\frac{n-1}{2}\right)!}{\sqrt{\pi}\cdot \Gamma(\frac{1}{2}+\frac{n+1}{2})}
		\item $a_{1,2n+1}=2b_{2n+2}$,
		\item
		$
			\lim_{n\to\infty}\frac{a_{1,2n+1}}{2b_{2n+1}}=1.
		$
		\end{itemize}
	\end{lem}
\begin{proof}
	Because $t\mapsto |\cos(t)^{n}|$ is $\pi$-periodic, %a well-known property of the Laplace transform yields
		\begin{equation*}
		q\int_{0}^{\infty}e^{-qt}|\cos(t)^{n}|\ dt =\frac{q\int_{0}^{\pi}e^{-qt}|\cos(t)^{n}|dt}{1-e^{-q\pi}},
		\end{equation*}
		which goes to $\frac{1}{\pi} \int_{0}^{\pi}|\cos(t)^{n}|dt$ as $q\to0^{+}$. Furthermore,
		\begin{align*}
			2_{b_{2n+2}}=\frac{2}{\pi} \int_{0}^{\pi}|\cos(t)^{2n+2}|dt=\frac{1}{\pi} \int_{0}^{2\pi}\cos(t)^{2n+1}\cos(t)dt	
			%={}&\frac{1}{\pi}B\left(\frac{1}{2},\frac{n+1}{2}\right).
		\end{align*}
		equals $a_{1,2n+1}$ by the Fourier series of $\cos(t)^{2n+1}$, see Lemma \ref{le:FScosn}.\newline By the same lemma we have that for $n \geq 1$
		\begin{equation*}	\frac{a_{1,2n-1}}{a_{1,2n+1}}=\frac{2^{-2n+2}\binom{2n-1}{n}}{2^{-2n}\binom{2n+1}{n}}=\frac{(2n+1)2n}{4(n+1)n},
		\end{equation*}
		which goes to $1$ as $n\to\infty$. This implies that $\frac{a_{1,2n+1}}{2b_{2n+1}}$ goes to $1$ as
		\begin{equation*}
		a_{1,2n+1}=2b_{2n+2}\leq 2b_{2n+1} \leq 2b_{2n} = a_{1,2n-1}, \quad n\in\mathbb{N}.
		\end{equation*}
	\end{proof}

\begin{proof}[Proof of Theorem \ref{mainresult}]
	Let $r=1-2\varepsilon$ for some $\varepsilon>0$. By Lemma \ref{le:bn} we can choose $n\in\mathbb{N}$ such that 
		\begin{equation}\label{eq:18}
		r\frac{2b_{2n+1}}{a_{1,2n+1}}<1-\varepsilon.
		\end{equation}
Let us abbreviate $a_{1,2n+1}$ by $a_{2n+1}$. By (\ref{eq:14}) and Proposition \ref{prop:rep}, we have that
	\begin{equation*}
	\left\|a_{2n+1}\frac{g(q)}{q}I+a_{2n+1}B(A,q)+G(A,q)\right\| \leq 2 r \int_0^{\infty} e^{-q t} |\cos(t)^{2n+1}| dt, 
	\end{equation*}
	hence,
	\begin{equation*}
	\left\|I+\frac{q}{g(q)}\left(B(A,q)+\frac{1}{a_{2n+1}}G(A,q)\right)\right\| \leq  \frac{2rq}{g(q)a_{2n+1} }\int_0^{\infty} e^{-q t} |\cos(t)^{2n+1}| dt, 
	\end{equation*}
	For $q\to0^{+}$, $g(q)\to1^{+}$, $qG(A,q)\to0$ by Proposition \ref{prop:rep} and by Lemma \ref{le:bn}, $q\int_{0}^{\infty}e^{-qt}|\cos(t)^{2n+1}|dt\to b_{2n+1}$. Thus, there exists $q_{0}>0$ (depending only on $\varepsilon$ and $n$) such that 
\begin{equation*}
	\left\|I+\frac{q}{g(q)}B(A,q)\right\| \leq r \frac{2b_{2n+1}}{a_{2n+1}}+\varepsilon=:\delta,\qquad \forall q\in(0,q_{0}),
	\end{equation*}
	Since $\delta<1$ by (\ref{eq:18}), $B(A,q)$ is invertible for $q\in(0,q_{0})$. Moreover,
	\begin{equation*}
		\|B(A,q)^{-1}\| \leq\frac{q}{g(q)}\cdot\frac{1}{1-\delta}.
	\end{equation*}
	Since
	\begin{align*}
	B(A,q)^{-1}={}&\frac{1}{2}((q-i)^{2}-A)q^{-1}\left[A-(q^{2}+1)I\right]^{-1}((q+i)^{2}-A),%\\
				%={}&-\frac{1}{2q}((q-i)^{2}-A)R(q^{2}+1,A)((q+i)^{2}-A),
	\end{align*}
	we conclude that 
	\begin{equation*}
	\|((q-i)^{2}-A)R(q^{2}+1,A)((q+i)^{2}-A)\|\leq \frac{q^{2}}{g(q)}\cdot\frac{2}{1-\delta}.
	\end{equation*}
	As $q\to0^{+}$, the right-hand-side goes to $0$, whereas the left hand side tends to $(I + A) 
    \left[ A -I \right]^{-1}(I + A)$ as $1\in\rho(A)$. Therefore, $A=-I$.
\end{proof}
	\begin{rem}
	We would like to address the open question whether condition (\ref{eq:10a1}) can be by replaced by
	\begin{equation*}
		\limsup_{t\to\infty}\|C(t)-\cos(at)\|=r,
	\end{equation*}
	for some $r<1$ (or even some constant larger than $1$). Even for the classical semigroup case this seems not so clear.
	\end{rem}


\begin{thebibliography}{9}

\bibitem{UlmerSeminare2012}
 W.~Arendt.
\newblock{\em A $0-3/2$ - {L}aw for {C}osine {F}unctions}. 
\newblock{ Ulmer Seminare, Funktional\-analysis und Evolutionsgleichungen}, 17:349--350, 2012. 

\bibitem{ABHN}
W.~Arendt, C.~J.~K. Batty, M.~Hieber, and F.~Neubrander.
\newblock {\em Vector-valued {L}aplace transforms and {C}auchy problems},
 Vol. 96, {Monographs in Mathematics}.
\newblock Birkh\"auser Verlag, Basel, 2001.

\bibitem{BobrowskiApprox}
A.~Bobrowski and W.~Chojnacki.
\newblock{\em  Isolated points of the set of bounded cosine families, bounded
 semigroups, and bounded groups on a {B}anach space.}
\newblock{ Studia Mathematica}, 217(3), 219-- 241, 2013.

\bibitem{BobrowskiZeroTwo}
A.~Bobrowski, W.~Chojnacki and A. Gregosiewicz.
\newblock{\em On close-to-scalar one-parameter cosine families.} Submitted.
%\newblock{\em J. Math. Anal. Appl.}, 2015.

\bibitem{ChojnackiZeroTwo}
W.~Chojnacki.
\newblock{\em Around Schwenninger and Zwart's Zero-two law for cosine families.} Submitted.
%\newblock{\em Studia Mathematica},
\bibitem{ChojnackiOneTwo}
W.~Chojnacki.
\newblock{\em On cosine families close to scalar cosine families. }
\newblock{ To appear in J. Aust. Math. Soc., arXiv:1411.0854.}

\bibitem{Esterle2015}
J.~Esterle.
\newblock{\em Bounded cosine functions close to continuous scalar bounded cosine functions.}
{\em arXiv:1502.00150.}

\bibitem{Fattorini69I}
H.~O. Fattorini.
\newblock{\em  Ordinary differential equations in linear topological spaces. {I}.}
\newblock { J. Differential Equations}, 5:72--105, 1969.

%\bibitem{Fattorini70cosHS}
%H.~O. Fattorini.
%\newblock{\em  Uniformly bounded cosine functions in {H}ilbert space.}
%\newblock {\ Indiana Univ. Math. J.}, 20:411--425, 1970/1971.

\bibitem{Hirschfeld}
R.~A. Hirschfeld.
\newblock{\em  On semi-groups in Banach algebras close to the identity.}
\newblock { Proc. Japan. Acad.}, 44:755, 1968.

\bibitem{Nagy74Sz}
B.~Nagy.
\newblock{\em On cosine operator functions in {B}anach spaces.}
\newblock { Acta Sci. Math. (Szeged)}, 36:281--289, 1974.

\bibitem{SchwenningerZwartZeroTwo}
F.~L.~Schwenninger, H.~Zwart.
\newblock{\em Zero-two law for cosine families.}
\newblock{ To appear in J. Evolution Equations}, 2015. {\em arXiv: 1402.1304.}
\bibitem{Staffans}
O.~Staffans.
\newblock{\em Well-posed linear systems}. Vol. 103,
 Encyclopedia of Mathematics and its Applications. 
\newblock Cambridge University Press, Cambridge, 2005.

\bibitem{Wallen}
L.~J. Wallen
\newblock{\em On the Magnitude  of $x^{n}-1$ in a Normed Algebra.}
\newblock { Proc. Amer. Soc.},18:956, 1967.


\end{thebibliography}
\end{document}